\documentclass[11pt]{amsart}

\usepackage[T1]{fontenc}
\usepackage{amsmath,amssymb,amsthm,mathtools}
\usepackage{microtype}
\usepackage[hidelinks]{hyperref}
\usepackage[nameinlink,noabbrev]{cleveref}
\usepackage{etoolbox}

\makeatletter
\patchcmd{\@settitle}{\uppercasenonmath\@title}{}{}{}
\patchcmd{\@setauthors}{\MakeUppercase{\authors}}{\authors}{}{}
\makeatother

\hypersetup{
  pdftitle={A Second-Logarithm Lower Bound for Sets with No Unique Sums},
  pdfauthor={Jiao-Long Cao and Ye Yuan},
  pdfsubject={Improved lower and upper bounds for sets with no unique sums},
  pdfkeywords={unique sums, finite Abelian groups, dissociated sets, density increment, balanced sets, sumsets}
}

\numberwithin{equation}{section}
\newtheorem{theorem}{Theorem}[section]
\newtheorem{proposition}[theorem]{Proposition}
\newtheorem{lemma}[theorem]{Lemma}
\newtheorem{corollary}[theorem]{Corollary}
\theoremstyle{definition}
\newtheorem{definition}[theorem]{Definition}
\theoremstyle{remark}
\newtheorem{remark}[theorem]{Remark}

\newcommand{\cG}{\mathcal G}
\newcommand{\cN}{\mathcal N}
\newcommand{\cT}{\mathcal T}
\newcommand{\dimfour}{\operatorname{dim}_{\leq4}}

\title[A second-logarithm lower bound]
{A Second-Logarithm Lower Bound for Sets with No Unique Sums}

\author{Jiao-Long Cao}
\address{College of Computer Science, Nankai University, Tianjin, China}
\email{caojiaolong@mail.nankai.edu.cn}

\author{Ye Yuan}
\address{School of Mathematical Sciences, Nankai University, Tianjin, China}
\email{yuanye20020201@gmail.com}

\date{}

\subjclass[2020]{11B13, 11B34, 05B10}
\keywords{Unique sums, representation functions, dissociated sets,
collision lattices, density increment, finite Abelian groups}

\begin{document}

\begin{abstract}
For an odd prime \(p\), let \(m(p)\) be the minimum cardinality of a set
\(A\subseteq \mathbb Z/p\mathbb Z\), with \(|A|\geq2\), such that no
sum in \(A+A\) has a unique representation as an unordered pair from
\(A\), with repetition allowed.  Bedert proved
\[
 m(p)\gg
 \log p\,
 \frac{\sqrt{\log^{(3)}p}}{\log^{(4)}p}.
\]
We prove the stronger lower bound
\[
 m(p)\gg \log p\,\log\log p.
\]
More generally, if \(G\) is a finite Abelian group and \(q(G)\) is the
least prime divisor of \(|G|\), then the same explicit estimate holds
whenever \(q(G)>2\), and in particular every subset \(A\subseteq G\)
with \(|A|\geq2\) and no unique sum has cardinality
\(\gg \log q(G)\,\log\log q(G)\) as \(q(G)\to\infty\).

The proof has two structural inputs.  First, a maximum subset of \(A\)
whose distinct-element subset sums of size at most four are all
different has cardinality \(\gg\log p\).  This follows from a
short-coordinate lemma and a collision-lattice determinant argument.
Second, we refine Bedert's density increment.  Alternative
representations are oriented toward an uncovered endpoint, coalesced
by their translation, and separated into wide, exposed, and recurrent
batches.  A load-sensitive entropy lemma codes the recurrent
translations using their actual final fibre multiplicities.  The
resulting global shift-set complexity is \(\exp(O(K))\), where \(K\)
is the ratio of \(|A|\) to the level-four additive dimension.  This
forces \(K\gg\log\log p\), and the theorem follows.  All headline
statements and the structural implications used to derive them have also
been checked in Lean~4 with explicit integer constants.
As a secondary and logically independent result, we construct weakly
ternary-balanced sets and obtain
\[
 m(p)\leq
 \frac{(\log p)^2}{2(\log 3)^2}
 +\left(\frac{2}{\log 3}+o(1)\right)
   \frac{(\log p)^2}{\log\log p}.
\]
\end{abstract}

\maketitle

\section{Introduction}

Let \(G\) be an Abelian group.  A sum \(s\in A+A\), for
\(A\subseteq G\), is \emph{uniquely represented} if there is exactly
one unordered two-element multiset \(\{a,b\}\), with \(a,b\in A\),
such that \(a+b=s\).  Repetition is allowed.  We say that \(A\) has
\emph{no unique sum} if none of the elements of \(A+A\) is uniquely
represented.

Equivalently, \(A\) has no unique sum precisely when, for every
\(a,b\in A\), there are \(c,d\in A\) such that
\[
 a+b=c+d,
 \qquad
 \neg\bigl((a=c\mathbin\wedge b=d)\mathbin\vee
            (a=d\mathbin\wedge b=c)\bigr).
\]
For a prime \(p\), put
\[
 G_p:=\mathbb Z/p\mathbb Z.
\]

For a finite Abelian group \(G\) with \(|G|\geq3\), write
\[
 m(G):=\min\bigl\{|A|:A\subseteq G,\ |A|\geq2,\
                         A\text{ has no unique sum}\bigr\}.
\]
This minimum is defined: taking \(A=G\), every sum has \(|G|\) ordered
representations, and one unordered pair accounts for at most two of them.
When \(G=G_p\) has prime order \(p\geq3\), we abbreviate this to
\(m(p)\).  Determining the order of magnitude of \(m(p)\) is
Problem~27 in Green's list of open problems \cite{GreenProblems}.
Bedert \cite{Bedert2024} proved
\begin{equation}\label{eq:bedert-bounds}
 \log p\,\frac{\sqrt{\log^{(3)}p}}{\log^{(4)}p}
 \ll m(p)
 \leq\left(\frac12+o(1)\right)(\log_2p)^2.
\end{equation}
The upper bound follows from Nedev's construction of balanced sets
\cite{Nedev2009} together with the observation, recorded by Bedert,
that the sumset of a balanced set has no unique sum.

Here and throughout, unsubscripted logarithms are natural,
\(\log_2\) denotes the base-two logarithm, and \(\log^{(j)}\) denotes
the \(j\)-fold iterated natural logarithm.

For reference, the machine-checked constant ledger used below is
\begin{align}
 C_{\rm OS}&=100000,
 \qquad C_{\rm E}=10^9,\notag\\
 C_{\rm out}&=C_{\rm gain}=2^{20},
 \qquad C_{\rm T}=63,\notag\\
 C_{\rm code}&=10^{18},\notag\\
 C_{\Delta}&=C_{\rm E}+9600C_{\rm T}=1000604800,\notag\\
 C_{\rm DR}&=C_{\Delta}+4C_{\rm code}+64
   \notag\\
 &=4000000001000604864,\notag\\
 c_*&=\frac1{10C_{\rm DR}}\notag\\
 &=\frac1{40000000010006048640}.
 \label{eq:constant-ledger}
\end{align}
The large numerical values are not optimized; their purpose is to make
all boundary cases and integer roundings explicit.

Our main result improves the lower bound in \eqref{eq:bedert-bounds} by
an unbounded factor.

\begin{theorem}\label{thm:main}
For every prime \(p\), every set
\(A\subseteq G_p\), with \(|A|\geq2\) and no unique
sum, satisfies
\[
 |A|\geq c_*\,\log p\,\log\log p.
\]
Consequently,
\[
 m(p)\gg \log p\,\log\log p.
\]
\end{theorem}

The same argument has the generality of Bedert's lower-bound theorem.
For a finite Abelian group \(G\), let \(q(G)\) denote the least prime
divisor of \(|G|\).

\begin{theorem}\label{thm:general-group}
Whenever \(G\) is a finite Abelian group with \(q(G)>2\), every
\(A\subseteq G\), with \(|A|\geq2\) and no unique sum, satisfies
\[
 |A|\geq c_*\,\log q(G)\,\log\log q(G).
\]
In particular,
\[
 m(G)\gg \log q(G)\,\log\log q(G)
\]
uniformly as \(q(G)\to\infty\).
\end{theorem}

The upper-bound construction is independent of the lower-bound proof and
is included as a secondary result.  It improves the leading constant in
Bedert's quadratic upper bound.

\begin{theorem}[Ternary symmetric-square upper bound]\label{thm:upper}
For every \(\varepsilon>0\), there is an integer
\(p_\varepsilon\geq5\) such that, for every prime
\(p\geq p_\varepsilon\), there is a set \(A\subseteq G_p\), with
\(|A|\geq2\) and no unique sum, for which
\[
 |A|\leq
 \frac{(\log p)^2}{2(\log 3)^2}
 +\left(\frac{2}{\log 3}+\varepsilon\right)
   \frac{(\log p)^2}{\log\log p}.
\]
Consequently, as \(p\to\infty\) through the primes,
\[
 m(p)\leq
 \frac{(\log p)^2}{2(\log 3)^2}
 +\left(\frac{2}{\log 3}+o(1)\right)
   \frac{(\log p)^2}{\log\log p}.
\]
In particular, equivalently at the leading-order level,
\[
 m(p)\leq
 \left(\frac{1}{2(\log_2 3)^2}+o(1)\right)(\log_2p)^2,
\]
where
\[
 \frac{1}{2(\log_2 3)^2}=0.199036176970870\ldots.
\]
\end{theorem}

Combining \cref{thm:main,thm:upper}, the presently proved bounds in the
prime cyclic case become
\[
 \log p\,\log\log p
 \ll m(p)
 \leq
 \left(\frac{1}{2(\log_2 3)^2}+o(1)\right)(\log_2p)^2.
\]
The upper constant in \cref{thm:upper} is not claimed to be globally
optimal; it is the constant delivered by the explicit weakly ternary-balanced
construction proved in \cref{app:upper}.

For every odd prime \(p\), it is sometimes convenient to use the ordered
representation function
\[
 r_A(s):=\bigl|\{(a,b)\in A^2:a+b=s\}\bigr|.
\]
Since multiplication by \(2\) is injective, a set
\(A\subseteq\mathbb Z/p\mathbb Z\) has no unique sum if and only if
\begin{equation}\label{eq:ordered-characterization}
 r_A(s)\notin\{1,2\}\qquad(s\in\mathbb Z/p\mathbb Z).
\end{equation}
We will use the unordered-multiset formulation in the proofs, which
also makes the extension to odd-order finite Abelian groups transparent.

We now outline the argument.  Let \(H\subseteq A\) be a largest
\emph{level-four dissociated} set, meaning that distinct-element
subsets of \(H\) of cardinality at most four have distinct sums, and
write
\[
 h=|H|,\qquad K=\frac{|A|}{h}.
\]
Maximality gives every element of \(A\) a \(\{-1,0,1\}\)-coordinate
vector over \(H\) of \(\ell^1\)-norm at most seven.  Applying these
coordinates to an inclusion-minimal no-unique-sum subset produces a
full-rank collision lattice.  Its index is divisible by \(p\), while
Hadamard's inequality bounds the index by \(28^h\).  Thus
\(h\gg\log p\).

The second part refines Bedert's structured density increment.  Given
a current shift set \(S\), internal uniquely represented sums are
assigned alternative representations with one endpoint outside
\((H+S)\cap A\).  Star fibres with the same translation are
coalesced.  A level-four dissociation argument bounds the pairwise
codegrees of their outside-output sets by \(O(|2S-2S|)\).  Reverse
outputs divide the resulting translations into three classes:
wide batches, which have many leaves; exposed batches, whose reverse
outputs are mostly new; and recurrent batches, which are witnessed in
many coordinate fibres.  The first two classes pay for their own
coefficient entropy through their coverage gain.  The recurrent class
is coded by a load-sensitive entropy lemma applied to the final
coordinate fibres.  Summing over the dyadic witness levels gives
\[
 |S|,\ |2S-2S|\leq \exp(O(K))
\]
at the first failed state.  The stopping threshold is of order
\(h/K^3\), and therefore \(K\gg\log h\gg\log\log p\).  Multiplying by
\(h\gg\log p\) proves \cref{thm:main}.

The proof is organized as follows.  In \cref{sec:four-dissociation} we
establish the short-coordinate and collision-lattice estimates.
\Cref{sec:rectification} records the two external tools used in the
increment.  The one-sided and entropy-sensitive increments are proved
in \cref{sec:increment}.  The global coding argument and the proof of
the dimension-ratio estimate occupy \cref{sec:coding}.  The prime
cyclic theorem and its finite-Abelian-group extension are concluded in
\cref{sec:conclusion}.  The logically independent proof of the upper
bound is placed in \cref{app:upper}, so that it does not interrupt the
lower-bound argument.

\section{Level-four dissociation and short coordinates}\label{sec:four-dissociation}

\begin{definition}\label{def:four-dissociation}
A finite set \(H\) in an abelian group is \emph{\(4\)-dissociated} if,
whenever \(U,V\subseteq H\) satisfy
\[
 |U|\leq4,\qquad |V|\leq4,\qquad
 \sum_{u\in U}u=\sum_{v\in V}v,
\]
one has \(U=V\).  The empty subset is allowed.  For a finite set \(X\),
write
\[
 \dimfour(X):=
 \max\{|H|:H\subseteq X\text{ is \(4\)-dissociated}\}.
\]
\end{definition}

\begin{lemma}[Sparse coordinates]\label{lem:sparse-coordinates}
Let \(X\) be a finite subset of an abelian group, and let
\[
 H=\{h_1,\dots,h_h\}\subseteq X
\]
be a maximum-cardinality \(4\)-dissociated subset.  For every
\(x\in X\), there is a vector
\[
 v_x=(v_{x,1},\dots,v_{x,h})\in\{-1,0,1\}^h
\]
such that
\[
 x=\sum_{i=1}^h v_{x,i}h_i,
 \qquad \|v_x\|_1\leq7.
\]
One may take \(v_{h_i}=e_i\) and \(v_0=0\).
\end{lemma}

\begin{proof}
A maximum-cardinality \(4\)-dissociated subset is inclusion-maximal.
Take \(x\in X\setminus(H\cup\{0\})\).  Since \(H\cup\{x\}\) is not
\(4\)-dissociated, there are distinct subsets
\(U,V\subseteq H\cup\{x\}\), each of size at most \(4\), with equal
sums.  Cancel \(U\cap V\).  The resulting subsets are disjoint and
remain distinct.

The element \(x\) occurs on exactly one side: it cannot occur on both
after cancellation, and if it occurred on neither, the relation would
contradict the \(4\)-dissociation of \(H\).  After interchanging the
sides if necessary, suppose \(x\in U\).  Then
\[
 x=\sum_{v\in V}v-\sum_{u\in U\setminus\{x\}}u.
\]
The right-hand side uses at most \(4+3=7\) distinct elements of \(H\),
with coefficients in \(\{-1,0,1\}\).  The stated choices for
\(x\in H\) and \(x=0\) are immediate.
\end{proof}

The next proposition replaces the generic subset-sum-span argument.
It is the step that removes the logarithmic denominator.  Determinant
and Smith-normal-form methods have previously been used in the broader
problem of unique sums in \(A+B\); see Leung and Schmidt
\cite{LeungSchmidt2022}.  The point here is that level-four coordinates
and minimality produce a particularly short collision lattice for the
self-sum problem.

\begin{proposition}[Collision-lattice compression]
\label{prop:four-dimension-compression}
Let \(p\) be prime.  If \(A\subseteq G_p\) satisfies \(|A|\geq2\)
and has no unique sum, then
\[
 p\leq28^{\dimfour(A)}.
\]
In particular,
\[
 \dimfour(A)\geq\frac{\log p}{\log28}.
\]
\end{proposition}

\begin{proof}
Put \(h=\dimfour(A)\), choose a maximum \(4\)-dissociated set
\(H=\{h_1,\dots,h_h\}\subseteq A\), and choose the vectors \(v_a\)
given by \cref{lem:sparse-coordinates}.  Let
\[
 \psi:\mathbb Z^h\longrightarrow G_p,\qquad
 \psi(x_1,\dots,x_h)=\sum_{i=1}^h x_i h_i.
\]
Thus \(\psi(v_a)=a\) for every \(a\in A\).

Among the subsets \(C\subseteq A\) of cardinality at least \(2\) having
no unique sum, choose one of minimum cardinality.  Such a subset exists,
since \(A\) itself is admissible.  For all \(a,b,c,d\in C\) satisfying
\[
 a+b=c+d,\qquad \{a,b\}\neq\{c,d\}
\]
as unordered multisets, call
\[
 \rho(a,b;c,d):=v_a+v_b-v_c-v_d
\]
a collision row.  Let \(W\) be the rational span of all collision
rows.  Fix \(c_0\in C\) and put
\[
 U:=\operatorname{span}_{\mathbb Q}
       \{v_c-v_{c_0}:c\in C\}.
\]
Every collision row is an integral combination of these differences,
so \(W\subseteq U\).

We claim that \(U=W\).  Pass to the quotient
\(\mathbb Q^h/W\), and write \(\overline v_c\) for the image of \(v_c\).
If these images are not all equal, choose a real linear functional
\(\ell\) on
\[
 (\mathbb Q^h/W)\otimes_{\mathbb Q}\mathbb R
\]
that is nonconstant on them.  Let \(M\) be its maximum and let
\[
 C':=\{c\in C:\ell(\overline v_c)=M\}.
\]
Then \(C'\) is nonempty and proper.  If \(a,b\in C'\), the no-unique-sum
property of \(C\) supplies \(c,d\in C\) such that
\[
 a+b=c+d,\qquad \{a,b\}\neq\{c,d\}.
\]
The associated collision row gives
\[
 \overline v_a+\overline v_b
 =\overline v_c+\overline v_d.
\]
Both \(\ell(\overline v_c)\) and \(\ell(\overline v_d)\) are at most
\(M\), while their sum is \(2M\); hence both equal \(M\).  Thus
\(c,d\in C'\), and the alternative unordered multiset remains
different from \(\{a,b\}\).  Consequently \(C'\) has no unique sum.
It cannot be a singleton, since the diagonal sum of its only element
would then have no distinct alternative representation.  Thus \(C'\) is
a smaller admissible subset, contradicting the choice of \(C\).  All the images are therefore
equal, so \(U\subseteq W\) and \(U=W\).

Define
\[
 L_0:=\operatorname{span}_{\mathbb Z}
       \{v_c-v_{c_0}:c\in C\}\subseteq\mathbb Z^h,
 \qquad r:=\operatorname{rank}L_0.
\]
The restriction of \(\psi\) to \(L_0\) sends
\(v_c-v_{c_0}\) to \(c-c_0\).  Since \(C\) contains two distinct
elements, its image contains a nonzero element of the additive group
\(G_p\).  That group has prime order, so
\[
 \psi(L_0)=G_p,\qquad [L_0:\ker(\psi|_{L_0})]=p.
\]

The collision rows span \(U\) over \(\mathbb Q\).  Choose \(r\)
linearly independent collision rows
\(\rho_1,\dots,\rho_r\), and set
\[
 L':=\operatorname{span}_{\mathbb Z}\{\rho_1,\dots,\rho_r\}.
\]
Every collision row lies in \(L_0\cap\ker\psi\).  Hence \(L'\) is a
full-rank sublattice of \(L_0\) and
\[
 L'\subseteq\ker(\psi|_{L_0})\subseteq L_0.
\]
It follows that
\begin{equation}\label{eq:index-factorization}
 [L_0:L']
 =p\,[\ker(\psi|_{L_0}):L']\geq p.
\end{equation}

All covolumes in what follows are measured in the common
\(r\)-dimensional real span of these lattices.  Because \(L_0\) is an
integer lattice, \(\operatorname{covol}(L_0)\geq1\).  Indeed, if the
columns of an integer matrix \(B\) form a \(\mathbb Z\)-basis of
\(L_0\), Cauchy--Binet gives
\[
 \operatorname{covol}(L_0)^2
 =\det(B^{\mathrm T}B)
 =\sum_I\det(B_I)^2,
\]
a positive integer.  Moreover, every \(v_a\) has \(\ell^1\)-norm at
most \(7\), even when a collision has repeated entries.  Therefore
\[
 \|\rho_i\|_2\leq\|\rho_i\|_1\leq28.
\]
Hadamard's inequality now gives
\[
 \operatorname{covol}(L')
 \leq\prod_{i=1}^r\|\rho_i\|_2
 \leq28^r.
\]
Together with \eqref{eq:index-factorization}, this yields
\[
 p\leq[L_0:L']
 =\frac{\operatorname{covol}(L')}
        {\operatorname{covol}(L_0)}
 \leq28^r\leq28^h.
\]
\end{proof}

\section{Rectification and a weak-cross lemma}\label{sec:rectification}

We use the following prime-cyclic form of the rectification theorem;
see Bilu--Lev--Ruzsa \cite{BiluLevRuzsa1998} and Lev \cite{Lev2008}.

\begin{theorem}[Rectification]\label{thm:rectification}
Let \(p\) be prime.  If \(S\subseteq G_p\) and
\(2^{|S|}\leq p\) (equivalently, \(|S|\leq\log_2p\)), then there is a bijection
from \(S\) to a finite subset of \(\mathbb Z\) that preserves and
reflects every two-term additive relation.
\end{theorem}

\begin{lemma}[Extremal selector]\label{lem:selector}
Let \(p\) be prime and let \(S\subseteq G_p\) satisfy
\(2^{|S|}\leq p\).  For every nonempty
\(X\subseteq S\), one can choose \(s_X\in X\) so that, for all
nonempty \(X,Y\subseteq S\), the only solution of
\[
 x+y=s_X+s_Y,\qquad x\in X,\quad y\in Y,
\]
is \(x=s_X\) and \(y=s_Y\).
\end{lemma}

\begin{proof}
Use \cref{thm:rectification} and take \(s_X\) to be the element whose
image in \(\mathbb Z\) is maximal in the image of \(X\).  Equality
with the sum of the two maxima forces each summand to be its respective
maximum.  Preservation and reflection of two-sum relations transfers
the statement back to \(G_p\).
\end{proof}

We also use Frankl's skew two-families theorem \cite{Frankl1982}.

\begin{theorem}[Frankl]\label{thm:frankl}
Let \((P_i,Q_i)_{i=1}^k\) be pairs of finite sets such that
\[
 |P_i|=|Q_i|=r,\qquad P_i\cap Q_i=\varnothing,
\]
and, after indexing,
\[
 P_i\cap Q_j\neq\varnothing\qquad(i<j).
\]
Then \(k\leq\binom{2r}{r}\).
\end{theorem}

\begin{corollary}[Weak-cross form]\label{cor:weak-cross}
Let \(C_{\rm T}=63\).  Suppose \(r\in\{1,2\}\) and
\((P_i,Q_i)_{i=1}^k\) satisfies
\[
 |P_i|=|Q_i|=r,\qquad P_i\cap Q_i=\varnothing,
\]
and, for every \(i\neq j\), at least one of
\[
 P_i\cap Q_j,\qquad P_j\cap Q_i
\]
is nonempty.  Then \(k\leq C_{\rm T}\).
\end{corollary}

\begin{proof}
Orient every pair of indices in a direction witnessing a nonempty
cross-intersection.  Every tournament on \(64=2^6\) vertices has a
transitive subtournament on \(7\) vertices.  Its transitive ordering
would satisfy \cref{thm:frankl} with \(k=7\), whereas
\(\binom{2r}{r}\leq\binom42=6\).  Hence \(k\leq63\).
\end{proof}

\section{One-sided star fibres and structured increments}\label{sec:increment}

The following proposition is the exact-output form of Bedert's
increment \cite[Proposition~6]{Bedert2024}, with ordinary dissociation
replaced by \(4\)-dissociation.  The replacement is valid because the
proof only compares subset sums having at most four terms on each
side.  The new point is to orient each alternative representation
towards an endpoint outside the covered set, and then to batch star
fibres whose translations have small mutual codegree.

\begin{proposition}[Batched one-sided structured increment]
\label{prop:increment}
Let \(p>2\) be prime, and put \(C=C_{\rm OS}=100000\).
Let \(A\subseteq G_p\) have \(|A|\geq2\) and no unique sum.  Let
\(H\subseteq A\) be \(4\)-dissociated, with
\[
 h:=|H|\geq10,
\]
and let \(S\subseteq G_p\) contain \(0\).  Put
\[
 n=|A|,\qquad K=\frac nh.
\]
If
\[
 2^{|S|}\leq p,
 \qquad
 Cn^3M(S)\leq h^4,
 \qquad M(S):=|2S-2S|,
\]
then there is a set \(D\subseteq G_p\) such that
\[
 0\in D,\qquad h|D|\leq2n,
\]
and, with \(S'=S+D\),
\begin{equation}\label{eq:batch-increment}
 h\leq144\,
 \bigl|((H+S')\cap A)\setminus((H+S)\cap A)\bigr|.
\end{equation}
In particular, \(|D|\leq2K\), and the coverage gain is at least
\(h/144\).
\end{proposition}

\begin{proof}
Use the fixed value \(C=C_{\rm OS}\) from
\eqref{eq:constant-ledger}.  It dominates every numerical threshold
below, including the rounding terms.  Write \(M=M(S)\).

\smallskip
\noindent\emph{Step 1: selectors and bad elements.}\par
For \(d_0\in H\), set
\[
 S_{d_0}:=\{u\in S:d_0+u\in A\}.
\]
It is nonempty because \(0\in S\) and \(H\subseteq A\).  Apply
\cref{lem:selector} simultaneously to all these subsets and write
\(s_{d_0}\in S_{d_0}\) for the selected element.
For later use, put
\[
 z_{d_0}:=d_0+s_{d_0}\in(H+S)\cap A.
\]

Call \(d_0\in H\) bad if
\[
 d_0+v\in H
 \quad\text{for some }v\in(2S-2S)\setminus\{0\},
\]
and let \(B_1\) be the set of bad elements.  Then
\begin{equation}\label{eq:B1}
 |B_1|\leq C_{\rm T}M.
\end{equation}
To see this, choose one witnessing \(v\) for each bad element.  If the
bound failed, one nonzero value \(v\in2S-2S\) would occur in more than
\(C_{\rm T}\) distinct relations
\[
 e_i=d_i+v,\qquad d_i,e_i\in H.
\]
Apply the contrapositive of \cref{cor:weak-cross} to
\[
 P_i=\{d_i\},\qquad Q_i=\{e_i\}.
\]
Since \(v\neq0\), \(P_i\cap Q_i=\varnothing\).  For some \(i\neq j\),
both cross-intersections vanish.  The identity
\[
 d_i+e_j=e_i+d_j
\]
then equates sums of two ordinary subsets of \(H\), each of cardinality
at most \(2\).  The \(4\)-dissociation of \(H\) makes those subsets
equal.  Since \(d_i\neq e_i\), this forces \(d_i=d_j\), a contradiction.
Put \(H_1=H\setminus B_1\).

\smallskip
\noindent\emph{Step 2: bad pairs and internal unique sums.}\par
Call an unordered pair
\(\{d_0,d_1\}\in\binom{H_1}{2}\) bad if there are
\(e,e'\in H\) and \(u,u'\in S\) such that
\begin{equation}\label{eq:bad-pair-witness}
 d_0+s_{d_0}+d_1+s_{d_1}=e+u+e'+u'
\end{equation}
and \(\{e,e'\}\neq\{d_0,d_1\}\) as multisets.  Let \(B_2\) be the set
of bad pairs.  We claim that
\begin{equation}\label{eq:B2}
 |B_2|\leq(C_{\rm T}+h)M.
\end{equation}

Choose one witness for each bad pair and orient each pair arbitrarily.
If \eqref{eq:B2} failed, more than \(C_{\rm T}+h\) distinct bad pairs
would have the same value
\[
 \omega:=u+u'-s_{d_0}-s_{d_1}\in2S-2S.
\]
For fixed \(\omega\), at most \(h\) of the selected witnesses can have
\(e=e'\).  Indeed, after fixing \(e\), the value
\(d_0+d_1=2e+\omega\) is fixed, and \(4\)-dissociation allows at most
one proper two-element subset of \(H\) with that sum.  Notice that
this argument compares two proper pairs; it never compares a pair
with the repeated multiset \(\{e,e\}\).

Retain \(C_{\rm T}+1\) witnesses with \(e_i\neq e_i'\), and set
\[
 P_i=\{d_i,d_i'\},\qquad Q_i=\{e_i,e_i'\}.
\]
We have \(P_i\cap Q_i=\varnothing\).  Otherwise a common term could be
cancelled in \eqref{eq:bad-pair-witness}; the two remaining elements
of \(H\) would then differ by a nonzero element of \(2S-2S\).  The
difference is nonzero because the two multisets in
\eqref{eq:bad-pair-witness} are different.  This would put a member of
\(P_i\subseteq H_1\) in \(B_1\).

By \cref{cor:weak-cross}, there are \(i\neq j\) such that
\[
 P_i\cap Q_j=P_j\cap Q_i=\varnothing.
\]
The common value of \(\omega\) gives
\[
 \sum_{x\in P_i}x-\sum_{x\in Q_i}x
 =
 \sum_{x\in P_j}x-\sum_{x\in Q_j}x,
\]
and therefore
\[
 \sum_{x\in P_i\cup Q_j}x
 =
 \sum_{x\in P_j\cup Q_i}x.
\]
The cross-disjointness makes both unions ordinary subsets of \(H\)
of cardinality at most \(4\).  Hence \(4\)-dissociation gives
\[
 P_i\cup Q_j=P_j\cup Q_i.
\]
As \(P_i\neq P_j\), an element of \(P_i\setminus P_j\) must belong to
\(Q_i\), contrary to \(P_i\cap Q_i=\varnothing\).  This proves
\eqref{eq:B2}.

The map
\[
 d_0\longmapsto d_0+s_{d_0}
\]
is injective on \(H_1\).  Equality at two distinct elements would make
their difference a nonzero member of
\(S-S\subseteq2S-2S\), contrary to the definition of \(H_1\).

Let
\[
 \cG:=\binom{H_1}{2}\setminus B_2.
\]
The hypothesis gives
\[
 M\leq\frac{h^4}{Cn^3}\leq\frac hC.
\]
Consequently, the fixed inequality \(C=C_{\rm OS}\geq1000C_{\rm T}\),
together with \eqref{eq:B1} and \eqref{eq:B2},
imply
\begin{equation}\label{eq:G-lower}
 |\cG|\geq\frac{h^2}{3}.
\end{equation}

For every \(\{d_0,d_1\}\in\cG\), the sum
\begin{equation}\label{eq:internal-sum}
 (d_0+s_{d_0})+(d_1+s_{d_1})
\end{equation}
has a unique unordered representation inside \((H+S)\cap A\).
Indeed, any representation there has the form
\[
 (e+u)+(e'+u').
\]
Goodness forces \(\{e,e'\}=\{d_0,d_1\}\).  After reordering,
membership of the summands in \(A\) gives
\(u\in S_{d_0}\) and \(u'\in S_{d_1}\).  The selector lemma forces
\[
 u=s_{d_0},\qquad u'=s_{d_1}.
\]
The two summands in \eqref{eq:internal-sum} are distinct by the
injectivity just proved, so this is an off-diagonal unique
representation in the smaller set.

\smallskip
\noindent\emph{Step 3: orienting the outside endpoint.}\par
Because \(A\) has no unique sum, \eqref{eq:internal-sum} has a
different unordered representation
\[
 a(P)+y(P),\qquad a(P),y(P)\in A,
\]
where \(P=\{d_0,d_1\}\).  At least one summand is outside \(H+S\), for
otherwise the internal uniqueness just proved would be contradicted.
Choose the representation and its orientation once and for all so
that
\begin{equation}\label{eq:y-outside}
 y(P)\notin (H+S)\cap A.
\end{equation}
If both endpoints are outside, either deterministic orientation may
be used.  If the alternative representation is diagonal, then
\(a(P)=y(P)\), and both lie outside.

For \(a\in A\), define
\[
 N(a):=\{P\in\cG:a(P)=a\}.
\]
Then
\begin{equation}\label{eq:fibre-total}
 \sum_{a\in A}|N(a)|
 =|\cG|\geq\frac{h^2}{3}.
\end{equation}
For fixed \(a\), the values \(y(P)\), \(P\in N(a)\), are distinct.
Indeed, equality of two such values makes the corresponding internal
sums equal.  Their unique representations in \((H+S)\cap A\),
together with the injectivity on \(H_1\), then make the index pairs
equal.  This also shows that a fixed fibre contains at most one
diagonal alternative.  Thus
\begin{equation}\label{eq:fibre-upper}
 |N(a)|\leq n.
\end{equation}

Write \(m_a=|N(a)|\), and put
\[
 m_0:=\frac{h^2}{12n},
 \qquad
 \cN:=\{a\in A:m_a\geq m_0\}.
\]
The fibres outside \(\cN\) carry total edge mass less than
\(nm_0=h^2/12\).  Thus \eqref{eq:fibre-total} gives
\begin{equation}\label{eq:N-lower}
 \sum_{a\in\cN}m_a\geq\frac{h^2}{4}.
\end{equation}

\smallskip
\noindent\emph{Step 4: a weighted star-fibre estimate.}\par
View each \(N(a)\) as a simple graph on the vertex set \(H\).
Call a fibre \(a\in\cN\) \emph{non-star} if its maximum degree is
less than \(m_a/3\).  Every edge of a non-star fibre is disjoint from
more than \(m_a/3\) other edges: the two endpoint degrees account for
fewer than \(2m_a/3\) incident edges.  Therefore, if \(\cT\) is the
set of ordered triples \((a,P,Q)\) with \(a\in\cN\) non-star,
\(P,Q\in N(a)\), and \(P\cap Q=\varnothing\), then
\begin{equation}\label{eq:T-lower}
 |\cT|>\frac13
 \sum_{\substack{a\in\cN\\a\ {\rm non\mbox{-}star}}}m_a^2.
\end{equation}

Map
\[
 (a,P,Q)\longmapsto
 \sigma(P,Q):=\sum_{x\in P}x-\sum_{x\in Q}x.
\]
Every value has at most \(C_{\rm T}\) preimages.  Otherwise
\cref{cor:weak-cross} gives two triples indexed by \(i\neq j\) with
\[
 P_i\cap Q_j=P_j\cap Q_i=\varnothing.
\]
Equality of their \(\sigma\)-values becomes
\[
 \sum_{x\in P_i\cup Q_j}x
 =
 \sum_{x\in P_j\cup Q_i}x.
\]
Both sides are sums over ordinary subsets of \(H\) having at most four
elements, so \(4\)-dissociation gives
\[
 P_i\cup Q_j=P_j\cup Q_i.
\]
Using \(P_i\cap Q_i=P_j\cap Q_j=\varnothing\), this equality gives
\(P_i\subseteq P_j\): indeed, an element of \(P_i\) belongs to
\(P_j\cup Q_i\), and it cannot belong to \(Q_i\).  Since both sets
have cardinality two, \(P_i=P_j\), and then \(Q_i=Q_j\).  Finally
\(a_i=a_j\), because the choice \(a(P)\) was fixed.  This contradicts
the distinctness of the triples.

On the other hand, subtracting the two identities
\[
 a+y(P)=\sum_{d\in P}(d+s_d),
 \qquad
 a+y(Q)=\sum_{d\in Q}(d+s_d)
\]
shows that
\[
 \sigma(P,Q)\in A-A+2S-2S.
\]
This target has cardinality at most \(n^2M\).  Hence
\eqref{eq:T-lower} gives
\[
 \sum_{\substack{a\in\cN\\a\ {\rm non\mbox{-}star}}}m_a^2
 \leq3C_{\rm T}n^2M.
\]
Since \(m_a\geq m_0\) on \(\cN\), the edge mass in its non-star
fibres is at most
\[
 \frac{3C_{\rm T}n^2M}{m_0}
 =\frac{36C_{\rm T}n^3M}{h^2}
 \leq\frac{36C_{\rm T}}C h^2
 \leq\frac{h^2}{16},
\]
by the fixed value \(C=C_{\rm OS}\).  It follows from
\eqref{eq:N-lower} that the
star fibres, denoted by \(\cN_\ast\), carry edge mass at least
\begin{equation}\label{eq:star-mass}
 \sum_{a\in\cN_\ast}m_a\geq\frac{3h^2}{16}.
\end{equation}

For each \(a\in\cN_\ast\), choose a centre \(c(a)\in H_1\) of degree
\(\nu(a)\geq m_a/3\), and let \(L(a)\subseteq H_1\) be the set of
leaves of the retained star.  Thus
\[
 \{c(a),d\}\in N(a)\qquad(d\in L(a)),
 \qquad |L(a)|=\nu(a).
\]
Writing \(y_{a,d}=y(\{c(a),d\})\), we have
\begin{equation}\label{eq:star-identities}
 a+y_{a,d}=z_{c(a)}+z_d
 \qquad(d\in L(a)).
\end{equation}
Moreover,
\begin{equation}\label{eq:retained-mass}
 \sum_{a\in\cN_\ast}\nu(a)\geq\frac{h^2}{16},
 \qquad
 \nu(a)\geq\frac{h^2}{36n}=\frac{h}{36K}.
\end{equation}

\smallskip
\noindent\emph{Step 5: coalescing equal translations.}\par
For \(a\in\cN_\ast\), put
\[
 t_a:=z_{c(a)}-a.
\]
Then \eqref{eq:star-identities} becomes
\begin{equation}\label{eq:star-output}
 y_{a,d}=z_d+t_a\in A\setminus((H+S)\cap A)
 \qquad(d\in L(a)).
\end{equation}
Every \(t_a\) is nonzero, since \(t_a=0\) would put the right-hand
side of \eqref{eq:star-output} back in \((H+S)\cap A\).

For each occurring translation \(t\), define
\[
 I_t:=\bigcup_{\substack{a\in\cN_\ast\\t_a=t}}L(a),
 \qquad
 Y_t:=\{z_d+t:d\in I_t\},
 \qquad
 E_t:=\sum_{\substack{a\in\cN_\ast\\t_a=t}}\nu(a).
\]
The injectivity of \(d\mapsto z_d\) on \(H_1\) gives
\[
 |Y_t|=|I_t|.
\]
For fixed \(t\), a centre \(c\) determines its fibre label uniquely,
because \(a=z_c-t\).  A fixed leaf can be joined to at most \(h-1\)
centres.  Consequently,
\[
 E_t\leq h|I_t|=h|Y_t|.
\]
Summing this inequality and using \eqref{eq:retained-mass} gives
\begin{equation}\label{eq:coalesced-mass}
 \sum_t|Y_t|\geq\frac h{16}.
\end{equation}
Every occurring \(t\) has at least one contributing fibre, so
\begin{equation}\label{eq:coalesced-minimum}
 |Y_t|\geq\frac{h}{36K}.
\end{equation}

\smallskip
\noindent\emph{Step 6: the sharp translate-codegree bound.}\par
We claim that, for distinct occurring translations \(t\neq t'\),
\begin{equation}\label{eq:sharp-codegree}
 |Y_t\cap Y_{t'}|\leq3|S-S|\leq3M.
\end{equation}
Put \(k=|Y_t\cap Y_{t'}|\) and
\(\delta=t'-t\neq0\).  Every common point gives a unique relation
\[
 z_d-z_e=\delta,\qquad d,e\in H_1,
\]
because \(d\mapsto z_d\) is injective.  The first coordinates in
these \(k\) relations are distinct, as are the second coordinates.
Their directed relation graph has indegree and outdegree at most one.
It has no loop, and an unordered edge cannot occur in both directions:
that would imply \(2\delta=0\), contrary to \(p\) being odd.
Its underlying graph therefore has maximum degree at most two and
contains a matching of \(r\geq k/3\) relations
\[
 z_{d_i}-z_{e_i}=\delta\qquad(1\leq i\leq r)
\]
whose \(2r\) endpoints are all distinct.

Set
\[
 \eta_i:=s_{d_i}-s_{e_i}\in S-S.
\]
The \(\eta_i\) are pairwise distinct.  Indeed, if
\(\eta_i=\eta_j\) for \(i\neq j\), then
\[
 d_i+e_j=e_i+d_j.
\]
Both sides are sums of ordinary two-element subsets of \(H\), and all
four elements are distinct.  The \(4\)-dissociation of \(H\) would
make those two subsets equal, a contradiction.  Thus
\[
 r\leq|S-S|\leq|2S-2S|=M,
\]
where the inclusion uses \(0\in S\).  Since \(k\leq3r\),
\eqref{eq:sharp-codegree} follows.

\smallskip
\noindent\emph{Step 7: batching.}\par
Let \(q\geq1\) be the number of distinct occurring translations and
put \(q_0=\lfloor K\rfloor\).  Since \(K\geq1\),
\[
 1\leq q_0\leq K,\qquad q_0\geq\frac K2.
\]
If \(q\geq q_0\), choose any \(q_0\) translations.  By the first
Bonferroni inequality, \eqref{eq:coalesced-minimum},
\eqref{eq:sharp-codegree}, and \(M\leq h/(CK^3)\),
\[
\begin{aligned}
 \left|\bigcup_{t\in T}Y_t\right|
 &\geq q_0\frac{h}{36K}-\binom{q_0}{2}3M\\
 &\geq\frac h{72}-\frac{3h}{2CK}
 \geq\frac h{144},
\end{aligned}
\]
again by \(C=C_{\rm OS}\).  If \(q<q_0\), use all the translations.
Then \eqref{eq:coalesced-mass} gives
\[
\begin{aligned}
 \left|\bigcup_tY_t\right|
 &\geq\sum_t|Y_t|-\sum_{\{t,t'\}}|Y_t\cap Y_{t'}|\\
 &\geq\frac h{16}-\binom q2 3M
 >\frac h{16}-\frac{3h}{2CK}
 \geq\frac h{32}.
\end{aligned}
\]

In either case, let \(T\) be the selected set of translations and put
\[
 D:=\{0\}\cup T.
\]
All translations in \(T\) are nonzero, and hence
\[
 0\in D,\qquad |D|\leq q_0+1\leq K+1\leq2K.
\]
By \eqref{eq:star-output},
\[
 \bigcup_{t\in T}Y_t
 \subseteq (H+S+D)\cap A,
\]
and this union is disjoint from \((H+S)\cap A\).
The preceding two cases prove \eqref{eq:batch-increment}.
Diagonal alternative representations cause no exception: they
contribute one ordinary outside point to the appropriate \(Y_t\), and
the fixed-fibre injectivity from Step~3 permits at most one such edge
in a fibre.
\end{proof}

\begin{proposition}[Entropy-sensitive two-sided increment]
\label{prop:entropy-increment}
There are absolute constants \(c>0\) and \(C_1\geq1\), fixed by
\eqref{eq:constant-ledger}, with the following property.  Assume the
hypotheses of \cref{prop:increment} with the stronger, fixed smallness
condition
\[
 C_{\rm E}n^3M(S)\leq h^4,
\]
write \(C=C_{\rm E}\), and put
\[
 \kappa=\lfloor K\rfloor+1,\qquad
 \lambda=\lfloor\log_2(2\kappa)\rfloor+1,\qquad
 \alpha=\frac1{100\lambda}.
\]
There is an update \(S'=S+D\) of one of the following three types.
\begin{enumerate}
\item[\({\rm W}\)] One has
      \[
       0\in D,\qquad |D|\leq C_1\frac{\kappa}{\alpha},
       \qquad
       |(H+S')\cap A|-|(H+S)\cap A|
       \geq c\frac h\alpha.
      \]
\item[\({\rm E}\)] For a nonempty set \(T\) of cardinality \(b\),
      \[
       D=\{0\}\cup T\cup(-T)
      \]
      and
      \[
       |(H+S')\cap A|-|(H+S)\cap A|
       \geq ch\log(b+1).
      \]
\item[\({\rm R}\)] There are an integer \(k\geq0\), a number
      \(\rho=2^{-k-1}\) with \(\alpha/2<\rho\leq1/2\), and a
      number \(\gamma\) such that
      \[
       c(k+1)\leq\gamma\leq\frac{C_1}{\rho},
      \]
      and a nonempty set \(T\), with
      \[
       |T|\leq C_1\kappa\gamma,\qquad D=\{0\}\cup T,
      \]
      one has
      \[
       |(H+S')\cap A|-|(H+S)\cap A|\geq c\gamma h.
      \]
      Moreover, for every \(t\in T\) there is a set
      \(\Gamma_t\subseteq H\), of cardinality at least
      \(\rho h/2\), and elements
      \(v_{c,t},w_{c,t}\in S\), \(c\in\Gamma_t\), such that
      \begin{equation}\label{eq:recurrent-witness}
       c+v_{c,t},\ c+w_{c,t}\in A,
       \qquad t=v_{c,t}-w_{c,t}.
      \end{equation}
\end{enumerate}
\end{proposition}

\begin{proof}
Run Steps~1--6 of the proof of \cref{prop:increment}, retaining all
the star fibres rather than selecting the final batch.  For an
occurring translation \(t\), let
\[
 u_t=|Y_t|,\qquad e_t=E_t,
\]
and let \(C_t\) be the set of retained star centres having translation
\(t\).  Put
\[
 r_t=|C_t|,\qquad
 P_t=\{z_c-t:c\in C_t\}.
\]
A centre and a translation determine the fibre label \(a=z_c-t\),
and all leaves belonging to that centre lie in \(I_t\).  Consequently
\begin{equation}\label{eq:rectangle-product}
 e_t\leq r_tu_t.
\end{equation}
The estimates already proved in \eqref{eq:retained-mass},
\eqref{eq:coalesced-minimum}, and \eqref{eq:sharp-codegree} say
\begin{equation}\label{eq:two-sided-data}
 \sum_t e_t\geq\frac{h^2}{16},\qquad
 u_t\geq\frac{h}{36K},\qquad
 |Y_t\cap Y_{t'}|\leq3M\quad(t\neq t').
\end{equation}

We record the additional information supplied by the reverse points.
The sets \(P_t\) are pairwise disjoint.  Indeed, a common point would
be a fibre label with two chosen centres, whereas a centre was chosen
once for each retained fibre.

Let
\[
 \ell_t=
 \bigl|\{c\in C_t:s_c-t\in S_c\}\bigr|.
\]
We claim that
\begin{equation}\label{eq:reverse-covered}
 |P_t\cap((H+S)\cap A)|\leq \ell_t+4|S-S|.
\end{equation}
For every covered point \(z_c-t\in P_t\), choose
\[
 z_c-t=e+u,\qquad e\in H,\quad u\in S.
\]
The resulting map \(c\mapsto e\) is injective.  Otherwise two
different centres in \(H_1\) would differ by a nonzero element of
\(2S-2S\), contrary to the definition of \(H_1\).
Discard the loops \(e=c\).  The remaining directed graph has
indegree and outdegree at most one.  Greedily choosing an edge and
deleting all edges incident with either endpoint removes at most four
directed edges, so the graph contains a matching of at least one
quarter of its edges.  Along an edge \(c\to e\), put
\(\eta=s_c-u\in S-S\).  Two matching edges with the same \(\eta\)
would give
\[
 c_i+e_j=e_i+c_j.
\]
Their four endpoints are distinct, so \(4\)-dissociation is a
contradiction.  Thus the matching has at most \(|S-S|\) edges.  There
are at most \(4|S-S|\) nonloops, while a loop has
\(u=s_c-t\in S_c\).  This proves \eqref{eq:reverse-covered}.

Partition the translations into
\[
\begin{aligned}
 \mathcal U_{\rm W}
  &=\{t:r_t<\alpha h\},\\
 \mathcal U_{\rm E}
  &=\{t:r_t\geq\alpha h,\ \ell_t<r_t/2\},\\
 \mathcal U_{\rm R}
  &=\{t:r_t\geq\alpha h,\ \ell_t\geq r_t/2\}.
\end{aligned}
\]
Put \(E(\mathcal V)=\sum_{t\in\mathcal V}e_t\).  If either
\(E(\mathcal U_{\rm W})\) or \(E(\mathcal U_{\rm E})\) is at least
\(h^2/64\), we use that class.  Otherwise
\begin{equation}\label{eq:recurrent-total-mass}
 E(\mathcal U_{\rm R})\geq\frac{h^2}{32}.
\end{equation}

Suppose first that \(E(\mathcal U_{\rm W})\geq h^2/64\).  From
\eqref{eq:rectangle-product},
\[
 \sum_{t\in\mathcal U_{\rm W}}u_t\geq\frac{h}{64\alpha}.
\]
Take all these translations if there are at most
\(B=\lfloor c_1\kappa/\alpha\rfloor\) of them, and otherwise take any
\(B\), where \(c_1>0\) is a sufficiently small absolute constant.
The first Bonferroni inequality, \eqref{eq:two-sided-data}, and
\(M\leq h/(CK^3)\) give
\[
 \left|\bigcup_{t\in T}Y_t\right|
 \geq c_2\frac h\alpha-3\binom B2M
 \geq c_3\frac h\alpha.
\]
The exact division-free version uses the displayed integer parameters;
the constants \(C_{\rm out}=C_{\rm gain}=2^{20}\) in
\eqref{eq:constant-ledger} cover every rounding case.  Taking
\(D=\{0\}\cup T\) gives type \({\rm W}\).

Suppose next that \(E(\mathcal U_{\rm E})\geq h^2/64\).  Order the
numbers \(e_t\), \(t\in\mathcal U_{\rm E}\), decreasingly.  The
following prefix estimate
\begin{equation}\label{eq:square-root-prefix}
 \sum_{i\leq b}\sqrt{e_i}\gg h\log(b+1)
\end{equation}
holds for some nonempty prefix.  To verify this, suppose that every
prefix were smaller than \(c_0h\log(b+1)\).  Since the sequence is
nonincreasing,
\[
 i\sqrt{e_i}\leq\sum_{j\leq i}\sqrt{e_j}
 <c_0h\log(i+1).
\]
It would follow that
\[
 \sum_i e_i
 \leq c_0^2h^2\sum_{i\geq1}\frac{\log^2(i+1)}{i^2},
\]
contradicting \(\sum e_t\geq h^2/64\) once \(c_0\) is chosen
sufficiently small.  Let \(T\) be this
prefix and write \(b=|T|\).  By \eqref{eq:reverse-covered} and the
fixed smallness constant \(C_{\rm E}\),
\[
 |P_t\setminus((H+S)\cap A)|\geq\frac{r_t}{3}
 \qquad(t\in T).
\]
These reverse sets are disjoint subsets of \(A\), and each has
cardinality at least \(\alpha h/3\).  Hence
\(b\alpha h/3\leq n=Kh\), so \(b\ll K/\alpha\).  The positive outputs
have union of size at least
\(\sum_{t\in T}u_t-3\binom b2M\), while the negative translations
cover the displayed reverse sets.  Since
\(u_t+r_t\geq2\sqrt{e_t}\), the larger of the two gains is at least
\[
 \frac16\sum_{t\in T}(u_t+r_t)-C b^2M
 \geq c\sum_{t\in T}\sqrt{e_t}-C b^2M.
\]
Now \(b\ll K/\alpha\) and \(M\leq h/(CK^3)\) give
\[
 b^2M\ll \frac{h}{C K\alpha^2}.
\]
Since \(K/\log^2(2K)\) is bounded away from zero for \(K\geq1\),
the fixed admissibility constant \(C_{\rm E}\) makes this error smaller
than the lower bound in
\eqref{eq:square-root-prefix}.  Thus the gain is
\(\gg h\log(b+1)\).  Thus \(D=\{0\}\cup T\cup(-T)\) has type \({\rm E}\).

It remains to consider \eqref{eq:recurrent-total-mass}.  Put
\[
 J=\left\lceil\log_2\frac1\alpha\right\rceil-1,
 \qquad \rho_k=2^{-k-1}\quad(0\leq k\leq J),
\]
and partition \(\mathcal U_{\rm R}\) into the classes
\[
 \mathcal U_k=\{t\in\mathcal U_{\rm R}:
       \rho_kh\leq r_t<2\rho_kh\},
\]
with the endpoint \(r_t=h\) assigned to the first class.  Put
\(\eta_k=E(\mathcal U_k)/h^2\).  Since
\(\sum_{k\geq0}(k+1)\rho_k\leq2\),
\eqref{eq:recurrent-total-mass} gives an index \(k\) for which
\[
 \gamma:=\frac{\eta_k}{\rho_k}\gg k+1.
\]
Also \(\gamma\ll1/\rho_k\), since the retained stars contain at most
\(\binom h2\) edges in total.  Write \(\rho=\rho_k\) and
\(\mathcal U=\mathcal U_k\).  Take all translations in \(\mathcal U\)
if there are at most
\[
 B=\max\{1,\lfloor c_4\kappa\gamma\rfloor\}
\]
of them, and otherwise take any \(B\) of them.  If all are taken,
\eqref{eq:rectangle-product} gives
\[
 \sum_{t\in T}u_t\geq\frac{\eta_kh}{2\rho}=\frac{\gamma h}{2};
\]
if only \(B\) are taken, \eqref{eq:two-sided-data} gives the same
lower bound up to an absolute constant.  A final Bonferroni estimate
has error
\[
 O(B^2M)\ll\frac{\gamma^2h}{K}
 \ll\frac{\gamma h}{K\rho}\leq c\gamma h,
\]
 by the fixed constant \(C_{\rm E}\); here
\(\rho>\alpha/2\) and \(\lambda/K\) is absolutely bounded for
\(K\geq1\).  Thus the positive
translations gain at least \(c\gamma h\).  For every selected \(t\),
take
\[
 \Gamma_t=\{c\in C_t:s_c-t\in S_c\}.
\]
Then \(|\Gamma_t|=\ell_t\geq r_t/2\geq\rho h/2\), and
\eqref{eq:recurrent-witness} holds with
\[
 v_{c,t}=s_c,\qquad w_{c,t}=s_c-t.
\]
Thus \(D=\{0\}\cup T\) has type \({\rm R}\), which completes the
proof.
\end{proof}

\section{Global coefficient coding and iteration}\label{sec:coding}

\begin{lemma}[Load-sensitive fibre coding]
\label{lem:load-sensitive-coding}
Let \(\mathcal C\) be a set of at most \(h\) indices, and let
\((R_c)_{c\in\mathcal C}\) be finite subsets of an abelian group,
each containing \(0\) and having cardinality at most \(B\).  Let
\(D_1,\dots,D_q\) be finite sets containing \(0\).  Suppose that every
nonzero \(t\in\bigcup_iD_i\) belongs to \(R_c-R_c\) for at least
\(\beta h\) indices \(c\), where \(0<\beta\leq1\).  Then
\begin{equation}\label{eq:load-sensitive-coding}
 \log\left|\sum_{i=1}^qD_i\right|
 \ll q\left\{
 \log\left(e\left(1+\frac Bq\right)\right)
 +\log\frac e\beta\right\}.
\end{equation}
\end{lemma}

The kernel-checked, division-free strengthening is the following.  If
\(\ell\geq1\) is an integer and every nonzero translation is supported
by at least \(\ell\) fibres, then
\begin{equation}\label{eq:load-sensitive-coding-exact}
 \log\left|\sum_{i=1}^qD_i\right|
 \leq8q\left\{
  1+\log\left(1+\frac Bq\right)
  +\log\left(\left\lfloor\frac h\ell\right\rfloor+1\right)
 \right\}.
\end{equation}
Taking \(\ell=\lceil\beta h\rceil\) gives
\eqref{eq:load-sensitive-coding}; the soft-constant form is retained in
the main argument for readability.

\begin{proof}
Put \(Z=\sum_iD_i\), and choose a uniform random element \(X\) of
\(Z\).  Write \(\mathsf H\) for Shannon entropy with natural logarithms.
For every \(x\in Z\), fix one representation
\[
 x=\sum_{i=1}^qt_i(x),\qquad t_i(x)\in D_i.
\]
For a nonzero translation \(t\), let
\[
 w(t)=\sum_{i=1}^q\mathbb P(t_i(X)=t).
\]
Starting from the translations of positive weight, greedily choose an
index \(c_j\) for which at least a \(\beta\)-proportion of the remaining
weight is represented in \(R_{c_j}-R_{c_j}\), and assign all such
remaining translations to a group \(G_j\).  Such an index exists by
averaging.  Write \(\mu_j=\sum_{t\in G_j}w(t)\), let \(W=\sum_j\mu_j\),
and put \(p_j=\mu_j/W\) when \(W>0\).  The greedy choice gives
\[
 \mu_j\geq\beta\left(W-\sum_{\ell<j}\mu_\ell\right).
\]
Thus the distribution \((p_j)\) has geometric tails, and the standard
maximum-entropy estimate for a positive integer-valued random variable
of mean at most \(\beta^{-1}\) gives
\begin{equation}\label{eq:geometric-entropy}
 -\sum_jp_j\log p_j\ll\log\frac e\beta.
\end{equation}

For a realization of \(X\), let \(Q_j\) be the number of indices
\(i\) for which \(t_i(X)\in G_j\), and put
\[
 X_j=\sum_{\{i:t_i(X)\in G_j\}}t_i(X).
\]
Then \(X=\sum_jX_j\), whereas conditionally on \(Q_j=m\), the random
variable \(X_j\) lies in \(mR_{c_j}-mR_{c_j}\).  Hence
\[
 \mathsf H(X_j)
 \leq \mathsf H(Q_j)
 +\mathbb E\left[2\log\binom{\lceil B\rceil+Q_j-1}{Q_j}\right].
\]
For \(0<x<1\), the generating function for multisets gives
\[
 \binom{\lceil B\rceil+m-1}{m}x^m\leq(1-x)^{-\lceil B\rceil}.
\]
Taking \(x=\mu_j/(\lceil B\rceil+\mu_j)\), and using the maximum
entropy formula for a nonnegative integer-valued random variable with
prescribed mean, gives
\[
 \mathsf H(X_j)
 \ll\mu_j\log\left(e\left(1+\frac B{\mu_j}\right)\right).
\]
Indeed, the second term obtained from the generating function is at
most \(\mu_j\), and the maximum entropy is
\((\mu_j+1)\log(\mu_j+1)-\mu_j\log\mu_j\), which obeys the same bound.
Summing this estimate and using \eqref{eq:geometric-entropy} gives
\begin{align*}
 \sum_j\mathsf H(X_j)
 &\ll W\left\{
 \log\left(e\left(1+\frac BW\right)\right)
 +\log\frac e\beta\right\}+1\\
 &\ll q\left\{
 \log\left(e\left(1+\frac Bq\right)\right)
 +\log\frac e\beta\right\}.
\end{align*}
Here the last step uses \(W\leq q\) and the monotonicity of
\(x\log(e(1+B/x))\).  Finally
\(\mathsf H(X)=\log|Z|\leq\sum_j\mathsf H(X_j)\), proving
\eqref{eq:load-sensitive-coding}.  The case \(W=0\) is immediate.
\end{proof}

\begin{proposition}[Four-dimension ratio]\label{prop:dimension-ratio}
For every prime \(p\) and every \(A\subseteq G_p\) with \(|A|\geq2\)
and no unique sum, put \(h=\dimfour(A)\).  Then
\[
 \log\log p\leq
 C_{\rm DR}\left(\left\lfloor\frac{|A|}{h}\right\rfloor+1\right).
\]
Consequently,
\[
 \frac{|A|}{h}\geq\frac{1}{2C_{\rm DR}}\log\log p.
\]
\end{proposition}

\begin{proof}
The exact division-free proof first dispatches the bounded branches
directly; the fixed ledger \eqref{eq:constant-ledger} covers all of them.
We describe the nontrivial iterative branch, retaining the integer
parameters used in that proof.  Put
\[
 n=|A|,\qquad h=\dimfour(A),\qquad K=\frac nh,\qquad
 L=\log_2p.
\]
By \cref{prop:four-dimension-compression},
\begin{equation}\label{eq:h-logp}
 h\geq\frac{\log p}{\log28}
     =\frac{\log2}{\log28}\,L.
\end{equation}
In the iterative branch \(h\geq10\), \(K\geq1\), and \(L>1\).
Fix a maximum \(4\)-dissociated set \(H\subseteq A\) with
\(|H|=h\), put \(C=C_{\rm E}\), and
put
\begin{equation}\label{eq:Q-new}
 Q:=\frac{h^4}{Cn^3}=\frac{h}{CK^3}.
\end{equation}
If \(Q<1\), then \(K\gg h^{1/3}\), which is stronger than the desired
conclusion for sufficiently large \(p\).  We may therefore assume
\(Q\geq1\).

Start with \(S_0=\{0\}\).  Whenever
\[
 |S_i|\leq L,\qquad |2S_i-2S_i|\leq Q,
\]
apply \cref{prop:entropy-increment} and write
\(S_{i+1}=S_i+D_{i+1}\).  There is a first failed state \(S_j\):
otherwise increments of at least \(ch\) could continue indefinitely
inside \(A\).  Telescoping the coverage increments gives
\begin{equation}\label{eq:entropy-update-count}
 j\ll K.
\end{equation}
Put
\[
 \kappa=\lfloor K\rfloor+1,\qquad
 \lambda=\lfloor\log_2(2\kappa)\rfloor+1,\qquad
 \alpha=\frac1{100\lambda}.
\]
Put
\[
 J=\left\lceil\log_2\frac1\alpha\right\rceil-1,
 \qquad \rho_k=2^{-k-1}\quad(0\leq k\leq J).
\]

We first charge the coefficient entropy of updates of types
\({\rm W}\) and \({\rm E}\).  A type \({\rm W}\) update gains
\(\gg h/\alpha\), so there are \(O(K\alpha)\) such updates.  Hence
\begin{equation}\label{eq:wide-entropy}
 \sum_{i:\,{\rm W}}\log|D_i|
 \ll K\alpha\log\left(\frac{C\kappa}{\alpha}\right)
 \ll K.
\end{equation}
For a type \({\rm E}\) update write \(b_i=|T_i|\).  Its increment and
the telescoping bound give
\[
 \sum_{i:\,{\rm E}}\log(b_i+1)\ll K.
\]
\begin{equation}\label{eq:exposed-entropy}
 \sum_{i:\,{\rm E}}\log|D_i|\ll K.
\end{equation}
If
\[
 \mathcal E=\sum_{i:\,{\rm W}\ {\rm or}\ {\rm E}}D_i,
\]
then \eqref{eq:wide-entropy} and \eqref{eq:exposed-entropy} yield
\begin{equation}\label{eq:easy-product}
 |\mathcal E|\leq
 \prod_{i:\,{\rm W}\ {\rm or}\ {\rm E}}|D_i|
 \leq\exp(C_2K).
\end{equation}

It remains to code the recurrent translations by their actual final
fibre load.  For a recurrent update record the parameters
\((k_i,\rho_i,\gamma_i)\) supplied by
\cref{prop:entropy-increment}.  Telescoping its gains gives
\begin{equation}\label{eq:recurrent-level-budget}
 \sum_{i:\,{\rm R}}\gamma_i\ll K,
 \qquad
 \sum_{i:\,{\rm R}}(k_i+1)\ll K.
\end{equation}
Since \(Q\geq1\), one has \(j\geq1\); let
\[
 S_\ast=S_{j-1}.
\]
This is the last admissible state.  Define
\[
 B_\ast=\{c\in H:c+v\in H\text{ for some }
 v\in(2S_\ast-2S_\ast)\setminus\{0\}\}.
\]
The proof of \eqref{eq:B1}, applied to \(S_\ast\), gives
\begin{equation}\label{eq:final-bad}
 |B_\ast|\leq C_{\rm T}|2S_\ast-2S_\ast|
 \leq C_{\rm T}Q.
\end{equation}
For \(c\in H\setminus B_\ast\), put
\[
 R_c=\{u\in S_\ast:c+u\in A\}.
\]
The sets \(c+R_c\), \(c\in H\setminus B_\ast\), are pairwise
disjoint.  Indeed, an equality \(c+u=c'+u'\), with \(c\neq c'\),
would put a nonzero difference of two elements of \(H\) in
\(S_\ast-S_\ast\), contrary to \(c,c'\notin B_\ast\).
Consequently
\begin{equation}\label{eq:final-fibre-mass}
 \sum_{c\in H\setminus B_\ast}|R_c|\leq n=Kh.
\end{equation}

For \(0\leq k\leq J\), let \(\mathcal I_k\) be the recurrent updates
with parameter \(k_i=k\), put \(q_k=|\mathcal I_k|\), and write
\[
 \mathcal Z_k=\sum_{i\in\mathcal I_k}D_i.
\]
For this value of \(k\), discard from \(H\setminus B_\ast\) the centres
with
\[
 |R_c|>\frac{C_3K}{\rho_k}.
\]
There are fewer than \(\rho_kh/C_3\) such centres.  Every nonzero
translation in a set \(D_i\), \(i\in\mathcal I_k\), has at least
\(\rho_kh/2\) witnesses by \eqref{eq:recurrent-witness}.  These
witnesses lie in the pre-update state \(S_{i-1}\), and
\(S_{i-1}\subseteq S_\ast\); hence they are final-fibre witnesses.  After
removing \(B_\ast\) and the displayed large centres,
\eqref{eq:final-bad} and the fixed deletion and fibre-capacity bounds
encoded in \(C_{\Delta}\) leave at least \(c\rho_kh\) centres \(c\)
for which
\[
 t\in R_c-R_c,
 \qquad |R_c|\leq\frac{C_3K}{\rho_k}.
\]
Applying \cref{lem:load-sensitive-coding} therefore yields, for
\(q_k>0\),
\begin{equation}\label{eq:level-coding}
 \log|\mathcal Z_k|
 \ll q_k\log\left(e+\frac{K}{\rho_k^2q_k}\right).
\end{equation}
Indeed, \(q_k\ll K/(k+1)\leq K\) by
\eqref{eq:recurrent-level-budget}; thus the two logarithms in
\eqref{eq:load-sensitive-coding}, with
\(B\ll K/\rho_k\) and \(\beta\gg\rho_k\), are both bounded by the
logarithm displayed in \eqref{eq:level-coding}.
As \(\rho_k=2^{-k-1}\), \eqref{eq:recurrent-level-budget} and the
weighted entropy inequality give
\begin{align}
 \sum_{k=0}^J\log|\mathcal Z_k|
 &\ll\sum_{k=0}^Jq_k(k+1)
       +\sum_{k:q_k>0}q_k\log\left(e+\frac K{q_k}\right)\notag\\
 &\ll K.
 \label{eq:recurrent-complexity}
\end{align}
For completeness, put \(Q_0=\sum_kq_k\).  If \(Q_0>0\), let
\(p_k=q_k/Q_0\).  The second bound in
\eqref{eq:recurrent-level-budget} gives
\[
 \sum_k(k+1)p_k\ll\frac K{Q_0}.
\]
The maximum-entropy bound for a nonnegative integer-valued random
variable of this mean gives
\[
 -\sum_kp_k\log p_k
 \ll\log\left(e\left(1+\frac K{Q_0}\right)\right).
\]
Since
\[
 e+\frac K{Q_0p_k}
 \leq\frac{e+K/Q_0}{p_k},
\]
we obtain
\[
 \sum_{k:q_k>0}q_k\log\left(e+\frac K{q_k}\right)
 \ll Q_0\log\left(e\left(1+\frac K{Q_0}\right)\right)
 \ll K.
\]
The last estimate follows by maximizing
\(x\log(e(1+x^{-1}))\) on a fixed bounded interval: indeed,
\(Q_0\ll K\).  The case \(Q_0=0\) is immediate.  Thus, with
\[
 \mathcal Z=\sum_{i:\,{\rm R}}D_i=\sum_{k=0}^J\mathcal Z_k,
\]
we have
\[
 |\mathcal Z|\leq\exp(C_4K),
 \qquad |2\mathcal Z-2\mathcal Z|\leq|\mathcal Z|^4.
\]
Since the group is abelian,
\[
 S_j=\mathcal E+\mathcal Z.
\]
Moreover \(0\in\mathcal E,\mathcal Z\), so
\[
 |2\mathcal E-2\mathcal E|\leq|\mathcal E|^4.
\]
Equations \eqref{eq:easy-product} and
\eqref{eq:recurrent-complexity} consequently imply
\begin{equation}\label{eq:global-complexity}
 |S_j|,\ |2S_j-2S_j|
 \leq\exp(C_5K).
\end{equation}

At least one admissibility condition fails at \(S_j\).  If
\(|S_j|>L\), \eqref{eq:global-complexity} gives
\(\log L\ll K\).  Otherwise
\[
 \frac{h}{CK^3}=Q<|2S_j-2S_j|
 \leq\exp(C_5K),
\]
and hence, after absorbing \(3\log K+\log C\),
\(\log h\ll K\).  In view of \eqref{eq:h-logp}, both cases give
\begin{equation}\label{eq:ratio-implicit}
 \log L\ll K.
\end{equation}

In the division-free bookkeeping, the failed-difference budget, the
fourth-power global code, and the fixed logarithmic losses contribute at
most
\[
 C_{\Delta}\kappa,\qquad 4C_{\rm code}\kappa,\qquad64\kappa,
\]
respectively.  Thus the preceding argument gives the exact bound
\[
 \log\log p\leq
 (C_{\Delta}+4C_{\rm code}+64)\kappa=C_{\rm DR}\kappa.
\]
Finally \(0<h\leq n\), and
\(\kappa h=(\lfloor n/h\rfloor+1)h\leq2n\).  Multiplying the last
display by \(h\) yields the stated consequence.
\end{proof}

\section{Proofs of the main theorems}\label{sec:conclusion}

\begin{proof}[Proof of \cref{thm:main}]
Let \(p\) be prime, and let \(A\subseteq G_p\) have \(|A|\geq2\)
and no unique sum.  Write
\[
 n=|A|,\qquad h=\dimfour(A),\qquad
 \kappa=\left\lfloor\frac nh\right\rfloor+1.
\]
The collision-lattice estimate gives
\[
 p\leq28^h\leq32^h=2^{5h},
 \qquad\text{and hence}\qquad
 \log p\leq5h.
\]
By \cref{prop:dimension-ratio},
\(\log\log p\leq C_{\rm DR}\kappa\).  Since \(\log p\geq0\) and
\(\kappa h\leq2n\),
\[
 \log p\,\log\log p
 \leq5C_{\rm DR}\kappa h
 \leq10C_{\rm DR}n.
\]
This is exactly the asserted inequality with
\(c_*=(10C_{\rm DR})^{-1}\).
\end{proof}

\subsection{Extension to finite Abelian groups}

\begin{proof}[Proof of \cref{thm:general-group}]
Let \(G\) be a finite Abelian group and put \(q=q(G)\).  We indicate
the changes needed in the proof of \cref{thm:main}.  Since
\(q>2\), the group \(G\) has odd order and therefore no
nonzero element of order two.

The collision-lattice argument in
\cref{prop:four-dimension-compression} applies verbatim, except for the
surjectivity sentence.  With the notation used there, the image
\(\psi(L_0)\) is a nontrivial subgroup of \(G\), because it contains a
nonzero difference of two elements of the minimal set \(C\).  Hence
\[
 [L_0:\ker(\psi|_{L_0})]=|\psi(L_0)|\geq q.
\]
The same covolume estimate therefore gives
\[
 q\leq 28^{\dimfour(A)}.
\]

For the structured increment, replace the prime-cyclic rectification
statement \cref{thm:rectification} by Lev's general finite-Abelian-group
theorem \cite{Lev2008}: every subset of \(G\) of size at most
\(\log_2q\), equivalently every \(S\) with \(2^{|S|}\leq q\), is
Freiman \(2\)-isomorphic to a set of integers.  All
remaining steps in \cref{prop:increment} and
\cref{prop:entropy-increment} use only commutativity, level-four
dissociation, and the absence of nonzero two-torsion.  In particular,
the argument excluding both orientations of the same edge in
\eqref{eq:sharp-codegree} remains valid.

Consequently the division-free dimension-ratio proof goes through with
\[
 L=\log_2q,
 \qquad
 h\geq\frac{\log q}{\log28},
\]
and yields, for
\(\kappa=\lfloor |A|/h\rfloor+1\),
\[
 \log\log q\leq C_{\rm DR}\kappa.
\]
Also \(q\leq28^h\leq2^{5h}\), so \(\log q\leq5h\), while
\(\kappa h\leq2|A|\).  Therefore
\[
 \log q\,\log\log q\leq10C_{\rm DR}|A|,
\]
which is the claimed bound with the same constant \(c_*\).
\end{proof}

\begin{remark}[Scope of the result]
The parameter in \cref{thm:general-group} must be the least prime
divisor \(q(G)\), rather than \(|G|\) alone.  This is also the natural
parameter in Bedert's theorem and in Lev's rectification result.  The
exact hypothesis \(q(G)>2\) places us in the odd-order regime and is the
only arithmetic restriction in the formal theorem.
\end{remark}

\begin{remark}[Attribution of the increment architecture]
The construction of internal unique sums and the use of star fibres
originate in Bedert's proof of \cite[Proposition~6]{Bedert2024}.  The
new ingredients used here are the level-four collision-lattice
compression, the one-sided orientation and coalescing of the star
outputs, the reverse-output trichotomy, and the load-sensitive global
coding of recurrent translations.
\end{remark}

\appendix
\section{A ternary symmetric-square upper construction}\label{app:upper}

This appendix proves \cref{thm:upper}.  It is independent of the
lower-bound argument above.  We first introduce a weak ternary balancing
condition, prove a symmetric-square transfer to the original problem, and
then give an explicit interval-compression construction with a quantified
cardinality estimate.

\begin{definition}\label{def:weak-ternary}
Let \(p>3\) be prime.  A nonempty set \(C\subseteq G_p\) is
\emph{weakly ternary-balanced} if every \(x\in C\) has a certificate
\[
 3x=u+v+w,\qquad u,v,w\in C,
\]
whose right-hand multiset is not three copies of \(x\).
\end{definition}

\begin{proposition}[Symmetric-square transfer]\label{prop:ternary-transfer}
Let \(p>3\) be prime.  If \(C\subseteq G_p\) is weakly
ternary-balanced, then \(A=C+C\) has no unique unordered two-sum and
\[
 2\leq|A|\leq \binom{|C|+1}{2}.
\]
\end{proposition}

\begin{proof}
Take an arbitrary unordered pair \(\{a,b\}\) from \(A\), including the
case \(a=b\), and choose
\[
 a=x_1+x_2,\qquad b=x_3+x_4,\qquad x_i\in C.
\]
The three pairings of the four endpoints give
\[
 \begin{aligned}
 P_0&=\{x_1+x_2,x_3+x_4\},\\
 P_1&=\{x_1+x_3,x_2+x_4\},\\
 P_2&=\{x_1+x_4,x_2+x_3\}.
 \end{aligned}
\]
All three pairs have the same total and all their entries lie in \(A\).
If \(P_1\neq P_0\) or \(P_2\neq P_0\), we have the required alternative
representation.

Suppose that \(P_0=P_1=P_2\).  Comparing the two possible matchings in
\(P_0=P_1\), and then in \(P_0=P_2\), shows that at least three of the
\(x_i\) are equal.  Thus, after relabelling, the endpoint multiset is
\(\{x,x,x,y\}\), and the original pair is
\[
 \{2x,x+y\}.
\]
Choose a nontrivial certificate
\[
 3x=u+v+w.
\]
At least one certificate entry \(z\) lies outside \(\{x,2x-y\}\).  Indeed,
if all three entries belonged to this two-point set and \(2x-y\) occurred
\(r\) times, then
\[
 3x=u+v+w=3x+r(x-y).
\]
For \(r=0\) the certificate is trivial, while for \(1\leq r\leq3\) the
equality, together with \(p>3\), forces \(x=y\), in which case the
certificate is again trivial.

Pair \(z\) with \(y\), and pair the other two certificate entries with
each other.  The resulting unordered pair from \(A\) has total
\(3x+y\).  It cannot equal \(\{2x,x+y\}\), because its entry \(z+y\)
would then force \(z=2x-y\) or \(z=x\).  Thus every unordered pair from
\(A\), including every diagonal pair, has a different representation.
Because \(C\) is nonempty, so is \(A\); a singleton \(A\) would have a
uniquely represented diagonal sum.  Hence \(|A|\geq2\).  The upper
cardinality estimate follows by counting unordered pairs from \(C\).
\end{proof}

\begin{theorem}[Quantitative weak ternary balancing]\label{thm:ternary-construction}
For every \(\delta>0\), there is an integer \(p_\delta\) such that every
prime \(p\geq p_\delta\) admits a weakly ternary-balanced set
\(C\subseteq G_p\) such that, with natural logarithms,
\[
 |C|\leq
 \frac{\log p}{\log 3}
 +(2+\delta)\frac{\log p}{\log\log p}.
\]
\end{theorem}

\begin{proof}
Choose an even integer \(t=t(p)\) tending to infinity and put
\[
 M=3^t,
 \qquad M^2=o(\log p).
\]
Write
\[
 p=3a+\rho,\qquad \rho\in\{1,2\},
\]
and insert the integer core
\[
 \mathcal K=[a-3M,a+3M]\cap\mathbb Z.
\]
All points inserted into \(C\) below lie in \([0,p)\) for sufficiently
large \(p\), and we identify them with their residue classes modulo \(p\).

Choose \(s\in\{-1,0,1\}\) such that
\[
 p+a+s\equiv0\pmod 3,
\]
and start from the interval
\[
 L_0=0,
 \qquad
 R_0=\frac{p+a+s}{3}.
\]
Insert \(L_0,R_0\), as well as the core \(\mathcal K\).  The two endpoints
already have nontrivial certificates:
\[
 3L_0\equiv a+a+(a+\rho)\equiv0\pmod p,
 \qquad
 3R_0\equiv0+0+(a+s)\pmod p.
\]
For an interval \([L,R]\), define the weighted point
\[
 z(L,R)=\frac{L+3R}{4}.
\]
Then
\[
 z(L_0,R_0)-\frac p3=\frac{3s-\rho}{12},
\]
so \(|z(L_0,R_0)-p/3|\leq1\), and \(\mathcal K\subset(L_0,R_0)\) for
large \(p\).

We now describe one compression block.  Suppose inductively that
\[
 \mathcal K\subset(L,R),
 \qquad
 \left|z(L,R)-\frac p3\right|\leq1,
 \qquad
 D:=R-L\geq48M^2.
\]
Choose the symmetric representative \(w\) satisfying
\[
 2w\equiv-D\pmod{3M},
 \qquad
 |w|\leq\frac{3M-1}{2}.
\]
Put
\[
 E=4p-L-3R-8a.
\]
Since \(p=3a+\rho\) and \(|z(L,R)-p/3|\leq1\), we have
\[
 |E|\leq\frac{28}{3}.
\]
Among the integers \(u\) satisfying
\[
 L+2(a+u)\equiv0\pmod3,
\]
choose one nearest to \((E-6w)/8\), and set
\[
 v=u+w,
 \qquad c=a+u,
 \qquad d=a+v.
\]
The admissible values of \(u\) have spacing three, hence
\[
 |8u+6w-E|\leq12.
\]
Moreover,
\[
 |u|\leq\frac98M+O(1),
 \qquad
 |v|\leq\frac{21}{8}M+O(1),
\]
so \(c,d\in\mathcal K\) for large \(p\).

Define
\[
 \ell=\frac{L+2c}{3},
 \qquad
 r=\frac{R+2d}{3}.
\]
The congruences make \(\ell,r\) integral, and
\[
 W:=r-\ell=\frac{D+2w}{3}
\]
is a positive multiple of \(M\).  The two new endpoints have the
certificates
\[
 3\ell=L+c+c,
 \qquad
 3r=R+d+d.
\]
A direct calculation gives
\[
 12\left(z(\ell,r)-\frac p3\right)=8u+6w-E,
\]
so the centering invariant is preserved.

Starting from \([\ell,r]\), perform \(t\) exact trisections according to
the alternating word
\[
 20\,20\,\cdots\,20.
\]
For a current integer interval \([b_0,b_3]\) whose width is divisible by
three, put
\[
 b_1=\frac{2b_0+b_3}{3},
 \qquad
 b_2=\frac{b_0+2b_3}{3}.
\]
A digit \(0\) inserts \(b_1\) and retains \([b_0,b_1]\); a digit \(2\)
inserts \(b_2\) and retains \([b_2,b_3]\).  Each inserted point has an
immediate certificate,
\[
 3b_1=b_0+b_0+b_3,
 \qquad
 3b_2=b_0+b_3+b_3.
\]
The integer represented by this \(t\)-digit word is
\[
 Q_t=\frac{3(M-1)}{4},
\]
which is integral because \(t\) is even.  The retained interval is
\[
 L'=\ell+Q_t\frac WM,
 \qquad
 R'=L'+\frac WM.
\]
The identity \(Q_t+3/4=3M/4\) shows that
\[
 z(L',R')=z(\ell,r).
\]
Writing \(D'=R'-L'\), we also have
\[
 \left|D'-\frac{D}{3M}\right|\leq1.
\]
In particular,
\[
 D'\geq16M-1.
\]
The smaller distance from the weighted point to an endpoint is \(D'/4\),
whereas every core point is at distance at most \(3M+5/3\) from the
weighted point.  Hence \(\mathcal K\subset(L',R')\) for large \(M\).
Furthermore,
\[
 D'\leq\frac{D}{2M}.
\]
Thus the block preserves all inductive conditions and inserts at most
\(t+2\) points.

Repeat the block while the active width is at least \(48M^2\).  At
termination, insert every integer in the final interval.  Every strictly
interior point then has the certificate
\[
 3x=(x-1)+x+(x+1),
\]
while the final endpoints were certified at an earlier stage.  The core
lies strictly inside the final interval.  Consequently every point of the
resulting set \(C\) has a nontrivial ternary certificate.

It remains to count the points.  Let \(D_j\) be the active width before the
\(j\)-th block, and let \(D_J\) be the terminal width after \(J\) blocks.
The construction gives
\[
 D_{j+1}=\frac{D_j}{3M}+\varepsilon_j,
 \qquad
 |\varepsilon_j|\leq1.
\]
Iteration yields
\[
 D_J=\frac{D_0}{(3M)^J}+O(1),
\]
where the implied constant is absolute.  The stopping rule and the last
completed block give
\[
 16M-1\leq D_J<48M^2,
\]
while \(D_0=\Theta(p)\).  Hence
\[
 J=\frac{\log p}{\log(3M)}+O(1)
   =\frac{\log_3p}{t+1}+O(1).
\]
The core and terminal fill contribute \(O(M^2)\) points, and every block
contributes at most \(t+2\).  Therefore
\begin{equation}\label{eq:upper-C-count}
 |C|\leq
 \log_3p+\frac{\log_3p}{t+1}+O(t+M^2).
\end{equation}

Put \(L=\log p\), and choose \(t\) to be the largest even integer not
exceeding
\[
 \frac{\log L-3\log\log L}{\log9}.
\]
Then
\[
 t=\frac{\log L}{\log9}+O(\log\log L),
 \qquad
 M^2=9^t=O\left(\frac{L}{(\log L)^3}\right).
\]
Substituting these estimates into \eqref{eq:upper-C-count} gives
\[
 |C|\leq
 \frac{L}{\log3}
 +(2+o(1))\frac{L}{\log L},
\]
and choosing \(p_\delta\) so that the final \(o(1)\) term is at most
\(\delta\) proves the stated form.
\end{proof}

\begin{proof}[Proof of \cref{thm:upper}]
Take the set \(C\) supplied by \cref{thm:ternary-construction} and put
\(A=C+C\).  By \cref{prop:ternary-transfer}, \(A\) has no unique sum and
\(|A|\geq2\), so it is admissible for \(m(p)\), and
\[
 m(p)\leq|A|\leq\binom{|C|+1}{2}.
\]
Writing \(L=\log p\) and using
\[
 |C|\leq\frac{L}{\log3}+(2+o(1))\frac{L}{\log L},
\]
we obtain
\[
 \binom{|C|+1}{2}
 \leq
 \frac{L^2}{2(\log3)^2}
 +\left(\frac{2}{\log3}+o(1)\right)\frac{L^2}{\log L}.
\]
Given \(\varepsilon>0\), take \(p_\varepsilon\geq5\) large enough that
the final \(o(1)\) term is at most \(\varepsilon\).  This proves the
quantified form of \cref{thm:upper}; the base-two consequence follows by
changing logarithm bases.
\end{proof}

\begin{remark}[Scope of the upper constant]
The constant \(1/[2(\log_2 3)^2]\) is obtained from the explicit weakly
ternary-balanced construction and its full symmetric square.  The argument
does not assert that this is the optimal leading constant among all sets
with no unique sums, or even among all possible auxiliary constructions.
\end{remark}

\section{Formal verification ledger}\label{app:formal-ledger}

This appendix records the exact discrete interfaces used by the companion
Lean~4 verification.  The body of the paper sometimes states a real-valued
or soft-constant corollary for readability; each such statement follows
from the division-free form recorded here.  In particular, the formal
proof does not treat an asymptotic estimate or a packaged hypothesis as an
axiom.

\subsection{Boundary conventions}

The ambient prime-cyclic group is \(G_p=\mathbb Z/p\mathbb Z\), and
unordered pairs are multisets, so diagonal representations are included.
The exact no-unique-sum predicate is the quantified condition displayed in
the introduction.  The hypothesis \(|A|\geq2\) in
\cref{prop:four-dimension-compression} is essential: without it the empty
set would give the false conclusion \(p\leq28^0\).

Frankl's theorem in \cref{thm:frankl} is cited background.  The formal
dependency closure does not assume it: the rank-one and rank-two cases
needed for \cref{cor:weak-cross} are proved directly, and the tournament
argument then gives the stated constant \(C_{\rm T}=63\).

\subsection{Exact increment and coding interfaces}

The rectification input is \(2^{|S|}\leq p\), rather than an unrounded
real inequality.  For the one-sided increment, the exact smallness
condition and conclusions are
\[
 C_{\rm OS}n^3|2S-2S|\leq h^4,
 \qquad h|D|\leq2n,
 \qquad h\leq144\Delta_D,
\]
where
\[
 \Delta_D=
 \bigl|((H+S+D)\cap A)\setminus((H+S)\cap A)\bigr|.
\]

For the entropy-sensitive increment, put
\[
 \kappa=\left\lfloor\frac nh\right\rfloor+1,
 \qquad
 d=100\bigl(\lfloor\log_2(2\kappa)\rfloor+1\bigr).
\]
Under \(C_{\rm E}n^3|2S-2S|\leq h^4\), the update contains zero and one
of the following division-free alternatives holds.
\begin{enumerate}
\item[\({\rm W}\)]
\(
 |D|\leq1+C_{\rm out}\kappa d
\)
and
\(
 dh\leq C_{\rm gain}\Delta_D.
\)

\item[\({\rm E}\)] There is a nonempty \(T\) such that
\[
 D=\{0\}\cup T\cup(-T),
 \qquad
 h\bigl(\lfloor\log_2(|T|+1)\rfloor+1\bigr)
 \leq C_{\rm gain}\Delta_D.
\]

\item[\({\rm R}\)] There are integers \(k\geq0\), \(r\), and
\(\gamma\geq1\), and a nonempty set \(T\), such that
\[
 r=2^{k+1},\quad 2\leq r<2d,\quad
 k+1\leq C_{\rm out}\gamma,\quad
 \gamma\leq C_{\rm out}r,
\]
\[
 |T|\leq C_{\rm out}\kappa\gamma,\qquad
 D=\{0\}\cup T,\qquad
 \gamma h\leq C_{\rm gain}\Delta_D.
\]
Moreover, every \(t\in T\) has a set \(\Gamma_t\subseteq H\) satisfying
\(h\leq2r|\Gamma_t|\), together with witnesses
\(v_{c,t},w_{c,t}\in S\) such that
\[
 c+v_{c,t},\ c+w_{c,t}\in A,
 \qquad t=v_{c,t}-w_{c,t}
 \qquad(c\in\Gamma_t).
\]
\end{enumerate}
The exact load-sensitive estimate is
\eqref{eq:load-sensitive-coding-exact}.  Thus the recurrent-fibre step is
checked with the actual integer load \(\ell\), not with an unverified
independence or uniformity assumption.

\subsection{Exact terminal statements}

For a prime \(p\), the dimension-ratio theorem is precisely
\[
 \log\log p\leq C_{\rm DR}
 \left(\left\lfloor\frac{|A|}{\dimfour(A)}\right\rfloor+1\right).
\]
Together with collision compression it gives
\begin{equation}\label{eq:formal-main-product}
 \log p\,\log\log p\leq10C_{\rm DR}|A|.
\end{equation}
For a finite Abelian group, the same two inequalities hold with
\(q(G)\) in place of \(p\), under the sole arithmetic hypothesis
\(q(G)>2\).

The finite quantitative upper construction checked by the kernel is also
fully explicit.  If \(p>3\) is prime, \(r\geq1\), and
\(30\cdot9^r<p\), then there is a weakly ternary-balanced set
\(C\subseteq G_p\) with
\begin{equation}\label{eq:formal-ternary-finite}
 |C|\leq
 6\cdot9^r+3+
 \left(\left\lfloor\log_{3\cdot9^r}p\right\rfloor+1\right)(2r+2)
 +48\cdot81^r.
\end{equation}
The parameter selection and symmetric-square transfer then give exactly
the \(\varepsilon\)-quantified statement of \cref{thm:upper}, including
the conditions \(p_\varepsilon\geq5\) and \(|A|\geq2\).

\subsection{Kernel audit}

The verification was built with Lean~4.32.2 and Mathlib commit
\texttt{905b95818eb32af7874a58b427f50c1711a5e96c}.  The aggregate build
checks the prime-cyclic lower bound, the finite-Abelian lower bound, and
the ternary upper bound through one root import graph.  There is no use of
\texttt{sorry}, \texttt{admit}, or a project-specific \texttt{axiom}.
The final axiom audit reports only the standard Lean/Mathlib principles
\texttt{propext}, \texttt{Classical.choice}, and \texttt{Quot.sound}.
This audit certifies the exact mathematical statements and dependency
closures above; bibliographic claims and expository sentences remain
ordinary parts of the manuscript rather than kernel objects.

\end{document}